\documentclass[11pt]{article}  
\usepackage{amssymb}  
\usepackage{amsthm}
\usepackage{amsmath}
\usepackage{tikz}
\usepackage{verbatim}
\usetikzlibrary{graphs,graphs.standard,quotes}
\usepackage{pgfplots}
\usepackage{xifthen}
\usetikzlibrary{calc}
\usetikzlibrary{positioning,chains,fit,shapes}
\usepackage{soul} 
\usepackage[T1]{fontenc}
\usepackage{authblk}
\usepackage{float}
\usepackage{mathtools,calc}
\usepackage{hyperref}
\newcounter{Angle}

\newtheorem{theorem}{Theorem}[section]
\newtheorem{corollary}[theorem]{Corollary}
\newtheorem{lemma}[theorem]{Lemma}

\newtheorem{proposition}[theorem]{Proposition}
\newtheorem{claim}[theorem]{Claim}
\newtheorem{fact}[theorem]{Fact}

\newtheorem{conjecture}[theorem]{Conjecture}

\tikzstyle{vertex}=[circle, draw, inner sep=0pt, minimum size=6pt]

\theoremstyle{definition}
\newtheorem{definition}[theorem]{Definition}

\begin{document}  
\author[1]{Jessica De Silva}
\author[2]{Xiang Si}
\author[2]{Michael Tait}
\author[3]{Yunus Tun\c{c}bilek}
\author[4]{Ruifan Yang}
\author[5]{Michael Young}
\affil[1]{University of Nebraska-Lincoln}
\affil[2]{Carnegie Mellon University}
\affil[3]{Yale University}
\affil[4]{Boston College}
\affil[5]{Iowa State University}
\title{Anti-Ramsey Multiplicities}

\maketitle

\begin{abstract}
The Ramsey multiplicity constant of a graph $H$ is the minimum proportion of copies of $H$ in the complete graph which are monochromatic under an edge-coloring of $K_n$ as $n$ goes to infinity. Graphs for which this minimum is asymptotically achieved by taking a random coloring are called {\em common}, and common graphs have been studied extensively, leading to the Burr-Rosta conjecture and Sidorenko's conjecture. Erd\H{o}s and S\'os asked what the maximum number of rainbow triangles is in a $3$-coloring of the edge set of $K_n$, a rainbow version of the Ramsey multiplicity question. A graph $H$ is called $r$-anti-common if the maximum proportion of rainbow copies of $H$ in any $r$-coloring of $E(K_n)$ is asymptotically achieved by taking a random coloring. In this paper, we investigate anti-Ramsey multiplicity for several families of graphs. We determine classes of graphs which are either anti-common or not. Some of these classes follow the same behavior as the monochromatic case, but some of them do not. In particular the rainbow equivalent of Sidorenko's conjecture, that all bipartite graphs are anti-common, is false.
\end{abstract}
\newcommand{\vertex}{\node[vertex]}

\section{Introduction}

All graphs that we consider will be finite and simple. If $H$ is a subgraph of $G$, we write $H \subseteq G$ and we say $G$ contains a \emph{copy} of $H$. An \emph{r-edge-coloring} of a graph $G$ is a function with domain $E(G)$ and codomain a set of $r$ colors, $\{1,\ldots,r\}$. Given an edge coloring $c$ of $G$, a subgraph $H$ of $G$ is said to be \emph{monochromatic} if for every $e,f\in E(H)$ $c(e)=c(f)$. 
That is, a subgraph is monochromatic if all its edges are the same color (e.g., Figure \ref{monochromatic}). 

\begin{figure}[b]
\centering
\begin{tikzpicture}[scale=0.7, every node/.style={scale=0.7}]
\node[circle,fill=black!20,draw] (A) at (0,-3) {1};
\node[circle,fill=black!20,draw] (B) at (3,-3) {4};
\node[circle,fill=black!20,draw] (C) at (0,0) {2};
\node[circle,fill=black!20,draw] (D) at (3,0) {3};

\draw (A) edge[blue,thick] (B);
\draw (C) edge[red,thick] (D);
\draw (A) edge[red,thick] (C);
 \draw (B) edge[blue,thick] (D);
 \draw (A) edge[red,thick] (D);
\end{tikzpicture}
\caption{The vertices $\{1,2,3\}$ form a monochromatic $K_3$.}\label{monochromatic}
\end{figure}
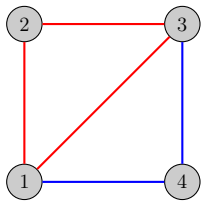

Given a complete graph $K_n$ and a subgraph $H$ of $K_n$, it is an interesting question to determine how many monochromatic copies of $H$ are we guaranteed to find in any $r$-edge-coloring of $K_n$. The maximum number we can guarantee is known as the Ramsey multiplicity. In particular, the \emph{Ramsey multiplicity} $M_r(H;n)$ is the minimum over all $r$-edge-colorings of $K_n$ of the number of monochromatic copies of $H$. We consider the Ramsey multiplicity of a graph $H$ with $m$ vertices relative to the number of copies of $H$ in $K_n$ via the ratio
\[C_r(H;n)=\frac{M_r(H;n)}{\binom{n}{m}\frac{m!}{|\textrm{Aut}(H)|}}.\]
The denominator is the number of copies of $H$ in $K_n$ where $\textrm{Aut}(H)$ is the set of automorphisms of $H$. Intuitively, this ratio can be thought of as the probability a randomly chosen copy of $H$ in $K_n$ is monochromatic. We can obtain an immediate bound on $C_r(H;n)$ by coloring each edge of $K_n$ color $i$ independently with probability $\frac{1}{r}$. Under this random coloring, any copy of $H$ in $K_n$ is monochromatic with probability $r^{1-e(H)}$. This gives an upper bound on $C_r(H;n)$ of $r^{1-e(H)}$. In \cite{jagger}, Jagger, \v{S}\v{t}ov\'i\v{c}ek, and Thomason show that $C_r(H;n)$ is nondecreasing in $n$ and so since it is also bounded the limit
\[C_r(H)=\lim_{n\to\infty} C_r(H;n),\]
exists and is known as the \emph{Ramsey multiplicity constant} of $H$ \cite{fox}.

The earliest result in this area was by Goodman in 1959 who proved $C_2(K_3)= \frac{1}{4}$ \cite{goodman}. In 1962, Erd\H{o}s conjectured that $C_2(K_n) = 2^{1-\binom{n}{2}}$ for all cliques \cite{erdos}. Burr and Rosta later conjectured that for all graphs $H$,  $C_2(H) = 2^{1-e(H)}$ \cite{burr}. We call a graph \textit{common} if it satisfies the Burr-Rosta conjecture. Sidorenko disproved the Burr-Rosta conjecture by showing that a triangle with a pedant edge is not common \cite{sidorenko}. Thomason disproved the initial conjecture of Erd\H{o}s by showing that for $p\geq 4$, $K_p$ is not common \cite{thomason}. Sidorenko conjectured instead that all bipartite graphs are common \cite{sidorenkobipartite}, this conjecture is well-known and is referred to as Sidorenko's conjecture. Much work has been done on the both the Burr-Rosta conjecture (see, e.g., \cite{jagger, goodman, burr, sidorenko, sidorenkofunc, jacobson}) and on Sidorenko's conjecture (c.f. \cite{blakleyroy, conlonfoxsudakov, hatami, kimleelee}). If we instead consider $r>2$, we call $H$ is called $r$\emph{-common} if $C_r(H) = r^{1-e(H)}$. Jagger et. al. showed that if a graph $G$ is not $r$-common, then it is not $(r+1)$-common \cite{jagger}. In 2011, Cummings and Young proved that no graph containing $K_3$ is 3-common \cite{young}. 
There are many open questions which remain for $r>2$.

We will consider a similar parameter to the Ramsey multiplicity constant by searching for rainbow subgraphs as opposed to monochromatic subgraphs. Given an edge coloring $c$ of $G$, a subgraph $H$ of $G$ is said to be \emph{rainbow} if for every pair of distinct edges $e,f\in E(H)$, $c(e)\neq c(f)$. 
In Figure \ref{monochromatic}, the edges $13$ and $34$ form a rainbow copy of $P_2$. Under this umbrella, a minimization problem is uninteresting since it is possible to color all edges the same color and hence contain no rainbow copy of $H$ (assuming $e(H)>1$). Instead, we ask what is the maximum number of rainbow copies of $H$ we can find amongst all edge colorings of $K_n$. Let $rb_r(H;n)$ be the maximum over all $r$-edge-colorings of $K_n$ of the number of rainbow copies of $H$ and call this the \emph{anti-Ramsey multiplicity} of $H$. In this paper, we will build the theory of the anti-Ramsey multiplicity constant and prove/disprove $r$-anti-commonality of various classes of graphs.

\section{The anti-Ramsey multiplicity constant}

Before we define the anti-Ramsey multiplicity constant, we will first prove that given a graph $H$, the maximum probability a copy of $H$ is rainbow under a coloring of $K_n$ is bounded and monotone as a function of $n$. As in the Ramsey case, we will consider the anti-Ramsey multiplicity of a graph $H$ with $m$ vertices relative to the number of copies of $H$ in $K_n$ via the ratio
\[rbC_r(H;n)=\frac{rb_r(H;n)}{\binom{n}{m}\frac{m!}{|\textrm{Aut}(H)|}}.\]
For the remainder of this section, fix a graph $H=(V,E)$ with $|V|=m$ and $e(H)=e$.

\begin{proposition}\label{bounded}
 \[rbC_r(H;n) \geq \frac{\binom{r}{e} e!}{r^e}\]
\end{proposition}

\begin{proof}
We will color the edges of $K_n$ uniformly and independently at random from the set $\{1,\ldots,r\}$. In particular, each edge is colored color $i$ with probability $\frac{1}{r}$ for $i=1,\ldots,r$. The number of possible rainbow edge assignments of a graph with $e$ edges is $\binom{r}{e}e!$ and a given edge assignment occurs with probability $\left(\frac{1}{r}\right)^e$. Thus the expected probability that a randomly selected copy of $H$ in $K_n$ is rainbow is given by $\frac{\binom{r}{e} e!}{r^e}$. Therefore there exists a coloring such that this probability is at least $\frac{\binom{r}{e}e!}{r^e}$ and since $rbC_r(C;n)$ is the maximum over all such probabilities, the inequality follows.
\end{proof}

\begin{proposition}\label{monotone}
\[rbC_r(H;n)\leq rbC_r(H;n-1)\]
\end{proposition}

\begin{proof}
The inequality is clear if $rbC_r(H;n)=0$ and so we suppose otherwise. Equivalently, we must show
\begin{align*}
\frac{rb_r(H;n) }{\binom{n}{m}} &\leq \frac{rb_r(H;n-1) }{\binom{n-1}{m}} \quad &\Longleftrightarrow\\
\frac{rb_r(H;n) }{\frac{n!}{m! (n-m)!}} &\leq \frac{rb_r(H;n-1) }{\frac{(n-1)!}{m! (n-m-1)!}}\quad &\Longleftrightarrow\\
\frac{rb_r(H;n) }{\frac{n}{n-m}} &\leq rb_r(H;n-1) \quad &\Longleftrightarrow\\
(n-m)rb_r(H;n)&\leq rb_r(H;n-1)n
\end{align*}
Let $c_r$ be an $r$-edge-coloring of $K_n$ such that the number of rainbow copies of $H$ in $K_n$ under coloring $c_r$ is exactly $rb_r(H;n)$. We will count the order of the set
\[H_n := \{(G,H)\,:\,G \textrm{ is a } K_{n-1} \subseteq K_n\textrm{ and }H \subseteq G \textrm{ is rainbow}\}\]
in two ways. First, note that each rainbow copy of $H$ is contained in $n-m$ different $K_{n-1}$ by removing any vertex in $K_n$ that is not a vertex of $H$. Since there are exactly $rb_r(H;n)$ copies of $H$ in $K_n$, $|H_n| = (n-m)rb_r(H;n)$. Now each $K_{n-1}$ in $K_n$ contains at most $rb_r(H;n-1)$ rainbow copies of $H$ and so $|G_n| \leq rb_r(G;n-1)n$. Therefore \[(n-m)rb_r(H;n)  = |H_n| \leq rb_r(H;n-1)n,\] which implies the result.
\end{proof}
We are now ready to define the anti-Ramsey multiplicity constant.

\begin{corollary}
The anti-Ramsey multiplicity constant, given by
\[rbC_r(H)=\lim_{n\to\infty}rbC_r(H;n),\]
exists and is finite.
\end{corollary}

\begin{proof}
By Propositions \ref{bounded} and \ref{monotone}, the sequence $\{rbC_r(H;n)\}_{n=m}^\infty$ is bounded and monotone. Hence by the Monotone Convergence Theorem, the limit exists and is finite.
\end{proof}
Note that the anti-Ramsey multiplicity constant has the same lower bound as that of Proposition \ref{bounded}, motivating the following definition.

\begin{definition}
For $r\geq m$, we say that $H$ is $r$-anti-common if
\[rbC_r(H)=\frac{\binom{r}{e}e!}{r^e}.\]
If $H$ is $r$-anti-common for all $r\geq m$, $H$ is called anti-common.
\end{definition}

\section{Anti-common graphs}

In this section we will prove anti-commonality for matchings and disjoint unions of stars. We will state but not prove the number of automorphisms for each graph in question and for more details regarding automorphisms of graphs see \cite{bona}. Suppose $f(n)$ and $g(n)$ are two real-valued functions. We say
\[f(n)=O(g(n))\]
if and only if there exist positive constants $C,N$ such that $|f(n)|\leq C|g(n)|$ for all $n>N$. We will sometimes abuse notation and use big-O notation in a string of inequalities. For example $f(n) \leq g(n) + O(n)$ means there exist $C,N$ such that $f(n) \leq g(n) + Cn$ for all $n\geq N$. 
\begin{lemma}\label{simpleineq}
If $H=(V,E)$ has order $m$ and size $e$ such that for sufficiently large $n$
\[rb_r(H;n)\leq \frac{n^m\binom{r}{e}e!}{|\textrm{Aut}(H)|r^e} + O(n^{m-1}),\]
then $H$ is $r$-anti-common.
\end{lemma}

\begin{proof}
Assume that for $n$ large enough we have $rb_r(H;n)\leq \frac{n^m\binom{r}{e}e!}{|\textrm{Aut}(H)|r^e} + O(n^{m-1})$. Then
\begin{align*}
\lim_{n\to\infty}\frac{rb_r(H;n)}{\binom{n}{m}\frac{m!}{|\textrm{Aut}(H)|}}&\leq\lim_{n\to\infty}\frac{\frac{n^m\binom{r}{e}e!}{|\textrm{Aut}(H)|r^e}+O(n^{m-1})}{\binom{n}{m}\frac{m!}{|\textrm{Aut}(H)|}}\\
&=\frac{\binom{r}{e}e!}{r^e}\lim_{n\to\infty}\frac{n^m + O(n^{m-1})}{\binom{n}{m}m!}\\
&=\frac{\binom{r}{e}e!}{r^e}\lim_{n\to\infty}\frac{n^m + O(n^{m-1})}{n(n-1)\cdots(n-m+1)}\\
&=\frac{\binom{r}{e}e!}{r^e}.
\end{align*}
\end{proof}
We will also use the following 
inequality, often referred to as Maclaurin's inequality.

\begin{fact}\label{macineq}
Given positive integers $r\leq n$ and positive real numbers $x_1,\ldots,x_n$,
\[\sum_{\{i_1,i_2,\ldots,i_r\}\subseteq [n]}x_{i_1}x_{i_2}\cdots x_{i_r}\leq \binom{n}{r}\left(\frac{\sum_{i=1}^n x_i}{n}\right)^r\]
\end{fact}

%
%

The following lemma will be used in the proof of Theorem \ref{disjointstars} which generalizes the result to disjoint unions of stars.

\begin{lemma}\label{stars}
Stars are anti-common.
\end{lemma}

\begin{proof}
Consider $S=K_{1,m-1}$ and note that
\[|\textrm{Aut}(S)|=(m-1)!.\]
By Lemma \ref{simpleineq}, It suffices to prove that for sufficiently large $n$, 
\[rb_r(K_{1,m-1};n) = \frac{\binom{r}{m-1}n^m}{r^{m-1}}+O(n^{m-1})\]
Given a vertex $v$ of $K_n$, let $q_i$ be the number of edges of color $i$ incident with $v$. Then the number of rainbow copies of $S$ with center $v$ is 
\[\sum_{\{i_1,i_2,\cdots, i_{m-1}\} \subseteq [r]} q_{i_1}q_{i_2}\cdots q_{i_{m-1}}.\]
Vertices of $K_n$ have degree $n-1$, so by Fact \ref{macineq} we have
\[\sum_{\{i_1,i_2,\cdots, i_{m-1}\} \subseteq [r]} q_{i_1}q_{i_2}\cdots q_{i_{m-1}}\leq\left(\frac{n-1}{r}\right)^{m-1}\binom{r}{m-1}.\]
Stars with centers $v$ and $v'$ are distinct if $v\neq v'$, therefore the total number of rainbow copies of $S$ in $K_n$ is at most
\[n\left(\frac{n-1}{r}\right)^{m-1}\binom{r}{m-1} = \frac{\binom{r}{m-1}n^m}{r^{m-1}}+O(n^{m-1}).\]
\end{proof}

\begin{theorem}\label{disjointstars}
Disjoint unions of stars are anti-common.
\end{theorem}

\begin{proof}
Fix positive integers $k\leq m$ and let $\mathcal{P}^{\geq 2}_k(m)$ denote the set of integer partitions of $m$ into $k$ parts with each part having size at least 2. For $P=\{\{m_1,\ldots,m_k\}\}\in\mathcal{P}_k^{\geq 2}(m)$, let $S_P$ be a disjoint union of $k$ stars with components $S_{P,i}=K_{1,m_i-1}$ for $i=1,\ldots,k$. Let $m_{i_1}\leq \cdots\leq m_{i_{j(P)}}$ be the $j(P)$ distinct sizes of the stars in $S_P$ and let $M_s$ be the number of stars in $S_P$ of size $m_{i_s}$. Then defining $\gamma(P)=\prod_{i=1}^{j(P)}M_i!$, we have the number of automorphisms of $S_P$ is given by
\[|\textrm{Aut}(S_P)|=\gamma(P)\prod_{i=1}^{k}(m_{i}-1)!.\]
Given $P\in\mathcal{P}_k^{\geq 2}(m)$, let
\[\binom{m-k}{P-1}=\binom{m-k}{m_1-1,\ldots,m_k-1}\]
then we want to show for sufficiently large $n$
\[rb_r(S_P;n)=\binom{m-k}{P-1}\frac{\binom{r}{m-k}\binom{n}{m}m!}{\gamma(P)r^{m-k}}+O(n^{m-1}).\]

\begin{claim}\label{starclaim}\[\sum_{P\in\mathcal{P}_k(m)}\gamma(P)rb_r(S_P;n) \leq \sum_{P\in\mathcal{P}_k(m)}\binom{m-k}{P-1}{\frac{\binom{n}{m} m!\binom{r}{m-k}}{r^{m-k}}}\]
\end{claim}

\begin{proof}
Let $\mathcal{C}_k(n)$ denote the collection of sets of $k$ distinguishable vertices in $K_n$. Given $C\in\mathcal{C}_k(n)$, we will count all the number of rainbow disjoint unions of $k$ stars with exactly $m$ vertices and with $C$ the set of centers. Let $q_i(C)$ denote the number of edges of color $i$ incident to any vertex in $C$, except those edges between two vertices in $C$. Then the number of rainbow disjoint unions of $k$ stars with $m$ vertices and distinguishable centers $C$ is exactly
\begin{equation}\label{starssum}\sum_{\{i_1,\ldots,i_{m-k}\}\subseteq [r]}q_{i_1}(C)\cdots q_{i_{m-k}}(C).\end{equation}
Note that $\sum_{i=1}^r q_i(C)=k(n-1)-\binom{k}{2}$ and so by Fact \ref{macineq} the sum in (\ref{starssum}) is at most
\[\binom{r}{m-k}\left(\frac{k(n-1)-\binom{k}{2}}{r}\right)^{m-k}.\]
The lefthand size of the inequality of this claim counts rainbow subgraphs such that given $P$. if $S_{P,i}$ and $S_{P,j}$ have the same order they will be distinguishable in the count above. Therefore since $|\mathcal{C}_k(n)|=\binom{n}{k}k!$, we have
\begin{align*}\sum_{P\in\mathcal{P}_k(m)}\gamma(P)rb_r(S_P;n)&\leq\binom{n}{k}k!\binom{r}{m-k}\left(\frac{k(n-1)-\binom{k}{2}}{r}\right)^{m-k}\\
&\leq \frac{\binom{r}{m-k}n^m}{r^{m-k}}k^{m-k}+O(n^{m-1})
\end{align*}
Notice that
\[\{\{\{m_1-1,\ldots,m_k-1\}\}\,:\,\{\{m_1,\ldots,m_k\}\}\in\mathcal{P}_k^{\geq 2}(m)\}\]
is the set of integer partitions of $m-k$ into $k$ parts. Therefore, by the Multinomial Theorem, we can rewrite
\begin{align*}
\frac{\binom{r}{m-k}n^m}{r^{m-k}}k^{m-k}+O(n^{m-1})&=\frac{\binom{r}{m-k}n^m}{r^{m-k}}\sum_{\{\{m_1,\ldots,m_k\}\}\in\mathcal{P}_k(m)}\binom{m-k}{m_1-1,\ldots,m_k-1}+O(n^{m-1})\\
&=\sum_{P\in\mathcal{P}_k^{\geq 2}(m)}\binom{m-k}{P-1}\frac{\binom{r}{m-k}\binom{n}{m}m!}{r^{m-k}}+O(n^{m-1})
\end{align*}
which proves the claim.
\end{proof}

By Proposition \ref{bounded}, we have for each $P=\{\{m_1,\ldots,m_k\}\}\in\mathcal{P}_k^{\geq 2}(m)$,
\begin{align}\gamma(P)rb_r(S_P;n)&\geq \frac{(m-k)!\binom{r}{m-k}\binom{n}{m}m!}{\prod_{i=1}^k(m_i-1)!r^{m-k}}+O(n^{m-1})\\
&=\binom{m-k}{P-1}\frac{\binom{r}{m-k}\binom{n}{m}m!}{r^{m-k}}+O(n^{m-1}).
\end{align}

Therefore, Claim \ref{starclaim} and the inequality (3) above implies for each $P\in\mathcal{P}_k^{\geq 2}(m)$,
\[rb_r(S_P;n)=\binom{m-k}{P-1}\frac{\binom{r}{m-k}\binom{n}{m}m!}{\gamma(P)r^{m-k}}.\]
\end{proof}

\section{Graphs which are not anti-common}\label{not anti-common section}

Not all graphs are $r$-anti-common for all $r$, and here we will prove in particular that complete graphs and $K_4$ without an edge are not anti-common. We will also give sufficient conditions, based on the number of edges, for a graph to not be anti-common.

\subsection{Specific graphs which are not anti-common}

In order to show that a graph is not anti-common for some $r$, we will construct a coloring with more rainbow subgraphs than that guaranteed in Proposition \ref{bounded}. Our arguments will start with a fixed coloring of some $K_m$ for $m$ small and we will use a blow-up argument to construct a coloring of a larger $K_n$.
\begin{definition}
A \emph{blow-up} is an inductive coloring of $K_n$, where the edges are colored as follows. Pick $m\leq n$ and fix a coloring of $K_m$ with labeled vertices $v_1,\ldots,v_m$. Divide the vertices of $K_n$ into $m$ disjoint sets of size $\lfloor\frac{n}{m}\rfloor$ and $\lceil\frac{n}{m}\rceil$, namely $V_1,\ldots,V_m$. For $u_i\in V_i$ and $u_j\in V_j$, color the edge $u_iu_j$ the same color as the edge $v_iv_j$ in the coloring of $K_m$. Repeat this process with each $V_i$ until there are no vertices left to be split into $m$ disjoint sets. We call this a blow-up of the initial coloring of $K_m$ with $n$ vertices.
\end{definition}
\begin{figure}[t]\label{k5coloring}
\centering
\begin{tikzpicture}[x=1.2cm, y=1.2cm]
	\foreach \i in {1, 2, 3, 4, 5} {
		\setcounter{Angle}{\i * 360 / 5 + 18}
		\vertex (\i) at (\theAngle:1) [fill=black]{};
	}
  \draw (1) edge[red, thick] (2);
  \draw (2) edge[blue, thick] (3);
  \draw (3) edge[green, thick] (4);
  \draw (4) edge[black, thick] (5);
  \draw (5) edge[yellow, thick] (1);
  
  \draw (3) edge[red, thick] (5);
  \draw (4) edge[blue, thick] (1);
  \draw (5) edge[green, thick] (2);
  \draw (1) edge[black, thick] (3);
  \draw (2) edge[yellow, thick] (4);
\end{tikzpicture}
\caption{A 5-edge-coloring of $K_5$ with 10 rainbow copies of $K_4\backslash e$.}
\end{figure}
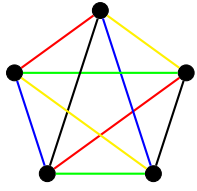

\begin{proposition}
The graph $K_4\backslash e$ is not 5-anti-common.
\end{proposition}

\begin{proof}
Note that the 5-edge-coloring of $K_5$ in Figure \ref{k5coloring} contains 10 rainbow copies of $K_4\backslash e$. Given $n=5^k$ for $k$ a positive integer, let $F(n)$ be the number of rainbow copies of $K_4\backslash e$ contained in a blow-up of the coloring in Figure \ref{k5coloring} on $n$ vertices. Within each of the 5 parts, there are $5F\left(\frac{n}{5}\right)$ rainbow copies of $K_4\backslash e$ and there are $10\left(\frac{n}{5}\right)^4$ with one vertex in each part. Therefore
\[F(n)\geq5F\left(\frac{n}{5}\right)+10\left(\frac{n}{5}\right)^4\]
and solving this recurrence gives
\[F(n)\geq \frac{n^4}{62}+O(n^3).\]
There are 4 automorphisms of $K_4\backslash e$, hence
\begin{align*}
rb_r(K_4\backslash e;n)&\geq\frac{n^4}{62}+O(n^3)\\
&>\frac{6n^4}{625}+O(n^3)\\
&=\frac{\binom{n}{4}4!\binom{5}{5}5!}{4\cdot 5^5}+O(n^3).
\end{align*}
\end{proof}

In \cite{rainbowtriangles}, it was shown that $K_3$ is not $3$-anti-common. We will now prove for $a\geq 4$, $K_a$ is not $\binom{a}{2}$-anti-common.

\begin{theorem}\label{complete}
The complete graph $K_a$ is not $\binom{a}{2}$-anti-common.
\end{theorem}

\begin{proof}
Consider a rainbow $K_a$, i.e. let $c$ be an $\binom{a}{2}$-edge-coloring of $K_a$ such that each edge is a different color. Given $n=a^k$ for $k$ a positive integer, let $F(n)$ denote the number of rainbow copies of $K_a$ contained in a blow-up of the coloring $c$ on $n$ vertices. There are $aF\left(\frac{n}{a}\right)$ rainbow copies of $K_a$ within each of the $a$ parts, and there are $\left(\frac{n}{a}\right)^a$ rainbow copies of $K_a$ with exactly one vertex from each part. Therefore
\[F(n)\geq aF\left(\frac{n}{a}\right)+\left(\frac{n}{a}\right)^a\]
and solving this recurrence gives
\[F(n)\geq\frac{n^a}{a^a-a}+O(n^{a-1}).\]
Therefore, since the number of automorphisms of $K_a$ is $a!$, in order to show
\[\frac{n^a}{a^a-a}+O(n^{a-1})>\frac{\binom{n}{a}\binom{a}{2}!}{\binom{a}{2}^{\binom{a}{2}}}\]
we will prove
\begin{equation}\label{factorial}
\frac{a!}{a^a-a}>\frac{\binom{a}{2}!}{\binom{a}{2}^{\binom{a}{2}}}.\end{equation}
We will use the following bounds on the factorial function
\[e\left(\frac{\binom{a}{2}}{e}\right)^{\binom{a}{2}}\leq \binom{a}{2}!\leq e\binom{a}{2}\left(\frac{\binom{a}{2}}{e}\right)^{\binom{a}{2}}\]
where $e$ is the base of the natural logarithm. From this we have
\[
\frac{\binom{a}{2}!}{\binom{a}{2}^{\binom{a}{2}}}\leq \frac{\binom{a}{2}}{e^{\binom{a}{2}-1}}
\]
and also using the inequality from (\ref{factorial}), $\frac{a!}{a^a-a}\geq\frac{1}{e^{a-1}}$ and therefore it's enough to show
\[\frac{\binom{a}{2}}{e^{\binom{a}{2}-1}}<\frac{1}{e^{a-1}}.\]
One can check that this inequality holds for $a\geq 4$ which concludes the proof.
\end{proof}

\subsection{Sufficient conditions for not anti-commonality}

In what follows $\log$ represents the natural logarithm. We will also be using both sides of the Stirling's approximation given below.
\begin{theorem}[Stirling's Approximation]
\[\sqrt{2\pi n}\left(\frac{n}{e}\right)^n\leq n!\leq \sqrt{2\pi n}\left(\frac{n}{e}\right)^ne^{\frac{1}{12n}}\]
\end{theorem}

\begin{theorem}\label{dense1}
Suppose $H$ is a graph with $m$ vertices and $e$ edges and let $c$ be a constant such that $2\pi m(1-c)>1$ and
\[c+(1-c)\log(1-c)\geq\frac{2}{m-1}+\frac{1}{\binom{m}{2}^212}.\]
If $e\geq c\binom{m}{2}$, then $H$ is not $\binom{m}{2}$-anti-common.
\end{theorem}

\begin{proof}
Let $H$ be a graph which satisfies the hypothesis above and consider a rainbow coloring of $H$. Blow-up this coloring to $n$ vertices and similar work as that in the proof of Theorem \ref{complete} gives that the number of rainbow copies of $H$ in $K_n$ is at least
\[\frac{n^mm!}{m^m}+O(n^{m-1}).\]
From the relationships between $c$ and $m$ we have
\[c\binom{m}{2}-\frac{1}{\binom{m}{2}12}+(1-c)\binom{m}{2}\log(1-c)-m\geq 0\]
and so raising both sides by the base of the logarithm $e$ gives
\[e^{c\binom{m}{2}-\frac{1}{\binom{m}{2}12}-m}(1-c)^{\binom{m}{2}(1-c)}\geq 1.\]
Then since $2\pi m(1-c)>1$ we have
\begin{align*}\sqrt{2\pi m(1-c)}e^{c\binom{m}{2}-\frac{1}{\binom{m}{2}12}-m}(1-c)^{\binom{m}{2}(1-c)}&>1\\
\frac{\sqrt{2\pi m}}{e^m}&>\frac{e^{\frac{1}{\binom{m}{2}12}}}{\sqrt{1-c}e^{c\binom{m}{2}}(1-c)^{\binom{m}{2}(1-c)}}\\
&=\frac{e^{\frac{1}{\binom{m}{2}12}}\left(\frac{\binom{m}{2}}{e}\right)^{\binom{m}{2}}}{\binom{m}{2}^{c\binom{m}{2}}\left(\frac{\binom{m}{2}(1-c)}{e}\right)^{\binom{m}{2}(1-c)}\sqrt{1-c}}\\
&\geq\frac{\binom{m}{2}!}{\binom{m}{2}^{c\binom{m}{2}}\left(\binom{m}{2}-c\binom{m}{2}\right)!\sqrt{1-c}}\\
&=\frac{\binom{\binom{m}{2}}{c\binom{m}{2}}\left(c\binom{m}{2}\right)!}{\binom{m}{2}^{c\binom{m}{2}}}\\
&\geq\frac{\binom{\binom{m}{2}}{e}e!}{\binom{m}{2}^{e}}.
\end{align*}
Using Stirling's approximation, we have
\[\frac{\sqrt{2\pi m}}{e^m}\leq \frac{m!}{m^m}.\]
and therefore
\[\frac{n^mm!}{m^m}+O(n^{m-1})>\frac{n^m\binom{\binom{m}{2}}{e}e!}{\binom{m}{2}^{e}}+O(n^{m-1})\]
\end{proof}

\begin{corollary}\label{dense2}
Let $H$ be a graph on $m$ vertices and $e$ edges such that 
\[e>m\sqrt{m-1}.\]
Then for $m\geq 6$, $H$ is not $\binom{m}{2}$-anti-common.
\end{corollary}

\begin{proof}
Let $H$ be a graph that satisfies the hypothesis and set $c=\frac{2}{\sqrt{m-1}}$. Since $2\pi m(1-c)>1$ for $m\geq 6$, we can apply Proposition \ref{dense1} and thus it suffices to show
\[c+(1-c)\log(1-c)\geq\frac{2}{m-1}+\frac{1}{\binom{m}{2}^212}.\]
For $m\geq 6$ we also have $|c|<1$, so we can expand the log function as follows
\begin{align*}
c+(1-c)\log(1-c)&=c+(1-c)\left(-c-\frac{c^2}{2}-\frac{c^3}{3}-\cdots\right)\\
&=\sum_{i=2}^\infty\frac{1}{i(i-1)}c^i\\
&=\frac{2}{m-1}+\frac{4}{3(m-1)^{3/2}}+\sum_{i=4}^\infty\frac{1}{i(i-1)}\left(\frac{2}{\sqrt{m-1}}\right)^i\\
&>\frac{2}{m-1}+\frac{1}{\binom{m}{2}^212}.
\end{align*}
\end{proof}

Corollary \ref{dense2} shows that for $n$ large enough, any bipartite graph of positive density is not anticommon. In particular, a random bipartite graph will satisfy the hypotheses of Corollary \ref{dense2} with probability tending to 1, giving the following corollary which is in sharp contrast to Sidorenko's conjecture. 

\begin{corollary}
Almost all bipartite graphs are not anti-common
\end{corollary}

If Sidorenko's conjecture is true, this is very different behavior from the monochromatic situation. 

\section{Future directions}

As in the Ramsey case, we wish to establish an implication between a graph being $r$-anti-common and $(r+1)$-anti-common. Through our investigation of this problem, we have shown the following inequality.

\begin{proposition}Let $H$ be a graph with $e$ edges, then
\[rb_{r+1}(H;n) \geq rb_r(H;n) \geq \left(\frac{(r+e)(r+1-e)}{r(r+1)}\right)rb_{r+1}(H;n).\]
\end{proposition}

\begin{proof}
Since the set of $(r+1)$-edge-colorings contains the set of $r$-edge-colorings, the left inequality follows immediately. Now consider an $(r+1)-$edge-coloring of $K_n$ such that the number of rainbow copies of $H$ is exactly $rb_{r+1}(H;n)$. Randomly choose a color from $[r+1]$ and call it $r'$. For all edges colored $r'$, recolor them randomly from the set of colors $[r+1]\backslash\{r'\}$. In the initial coloring, the expected number of rainbow copies of $H$ with one edge colored $r'$ is 
\[\frac{rb(G,n,r+1)e}{r+1}.\]
With probability $\frac{r-e+1}{r}$, each of these rainbow subgraphs will remain rainbow in the new coloring. Therefore the expected number of rainbow copies of $H$ in the new coloring is
\[\left(rb_{r+1}(H;n) - \frac{rb_{r+1}(H;n)e}{r+1}\right) + \frac{rb_{r+1}(H;n)e(r-e+1)}{r(r+1)} = \left(\frac{(r+e)(r+1-e)}{r(r+1)}\right) rb_{r+1}(H;n). \]
This implies that there exists such a coloring of $K_n$ with $r$ colors and hence
\[\left(\frac{(r+e)(r+1-e)}{r(r+1)}\right)rb_{r+1}(H;n)\leq rb_r(H;n).\]
\end{proof}

This inequality leads us to believe that the implication below is in fact true.

\begin{conjecture}
If $H$ is not $r$-anti-common, then $H$ is not $(r+1)$-anti-common.
\end{conjecture}

There are also many other classes of graphs whose anti-commonality have yet to be studied. Preliminary results on cycles lead us to believe that for $k\geq 3$, cycles of length $k$ are not $k$-anti-common. One can show using the blow-up method in Section \ref{not anti-common section} that $C_4$ is not $4$-anti-common and that $C_5$ is not $5$-anti-common. It is also conjectured that $P_4$ is 3-anti common---flag algebras (on 5 vertex flags) give an upper bound of approximately $0.22222241$, nearly matching the lower bound of $2/9$.

\section{Acknowledgments}

We would like to thank Carnegie Mellon University for supporting the Summer Undergraduate Applied Mathematics Institute. Additionally, we gratefully acknowledge financial support for this research from the following grants: NSF DGE-1041000 (Jessica De Silva), NSF DMS-1606350 (Michael Tait), and NSF DMS-1719841 (Michael Young).

\end{document}